\documentclass[draft]{amsart}
\usepackage{amssymb,latexsym,amsfonts,verbatim,amscd}
\usepackage[usenames,dvipsnames]{pstricks}
\usepackage{epsfig}
\numberwithin{equation}{section}

\theoremstyle{plain}

\newtheorem{theorem}[equation]{Theorem}
\newtheorem{lemma}[equation]{Lemma}
\newtheorem{proposition}[equation]{Proposition}
\newtheorem{corollary}[equation]{Corollary}
\newtheorem{construction}[equation]{Construction}

\theoremstyle{definition}

\newtheorem{definition}[equation]{Definition}
\newtheorem{remark}[equation]{Remark}
\newtheorem{remarks}[equation]{Remarks}
\newtheorem{notation}[equation]{Notation}
\newtheorem{example}[equation]{Example}

\DeclareMathOperator{\ann}{ann} \DeclareMathOperator{\supp}{supp}
\DeclareMathOperator{\Tor}{Tor}

 \DeclareMathOperator{\coker}{Coker}

 \DeclareMathOperator{\im}{Im}

\DeclareMathOperator{\lcm}{lcm} 
\DeclareMathOperator{\sgn}{sgn} 

\DeclareMathOperator{\Fitt}{Fitt}

\DeclareMathOperator{\diag}{diag}
\DeclareMathOperator{\join}{join}
\DeclareMathOperator{\link}{link} \DeclareMathOperator{\ini}{in}
\DeclareMathOperator{\ima}{Im}

\def\mapright#1{\smash{\mathop{\longrightarrow}\limits^{#1}}}
\def\mapdown#1{\smash{\mathop{\downarrow}\limits^{#1}}}

\renewcommand\dim{{\rm dim}}

\renewcommand\a{\alpha}
\renewcommand\b{\beta}
\renewcommand\d{\delta}
\renewcommand\t{\tau}
\renewcommand\th{\theta}

\newcommand{\ds}{\displaystyle}

\newcommand\D{\Delta}

\newcommand{\g}{\gamma}
\newcommand{\ro}{\rho}
\newcommand{\s}{\sigma}

\begin{document}

\title[Topological Constructions for Multigraded Squarefree Modules]
{Topological Constructions for Multigraded Squarefree Modules}
\author[H. Charalambous]{Hara Charalambous}
\address{Department of Mathematics\\
Aristotle University of Greece         \\
Thessaloniki, 54124         } \email{hara@math.auth.gr}
\keywords{} \subjclass{}

\begin{abstract} Let $R=\Bbbk[x_1,\ldots, x_n]$ and $M=R^s/I$ a
multigraded squarefree module. We discuss the construction of
cochain complexes associated to $M$ and we show how to interpret
homological invariants of $M$ in terms of topological
computations. This is a generalization of the well studied  case
of squarefree monomial ideals.

\end{abstract}

\maketitle



\section{Introduction}\label{S: introd}

Let $R=\Bbbk[x_1,\ldots, x_n]$ be a polynomial ring over the field
$\Bbbk$ of characteristic $0$. For  $\a=(a_i)\in \mathbb{Z}^n$
 we let
$x^\a=x_1^{a_1}\cdots x_n^{a_n}$ and $R_\a=kx^\a$. Let
$A=(c_{ij}x^{\a_{ij}})\in \mathbb{M}_{s\times l}(R)$: $c_{ij}\in
\Bbbk$ and $\a_{ij} \in \mathbb{N}^n$. We say that  $A$ is {\em
multigraded} if each minor of $A$ equals $c_\a x^\a$ for some
$c_\a\in \Bbbk$. We say that $A$ is of {\em uniform rank} if all
of its minors are nonzero. In particular this implies that
$c_{ij}\neq 0$ for all $i$ and $j$ and thus the matrix  of
coefficients $(c_{ij})$ is sufficiently generic. In addition we
say that $\a\in \mathbb{N}^n$ is  {\em squarefree} if $\a\in
\{0,1\}^n$ and a collection of vectors is {\em squarefree} if
each vector is squarefree.

 We recall that $M$ is a multigraded $R$-module if
$M=\oplus_{\a\in \mathbb{N}^n} M_\a $ where $M_\a$ are subgroups
of $M$ and $R_{\a_1} M_{\a_2} \subset M_{\a_1+\a_2}$ whenever
$\a_i\in \mathbb{N}^n$. Moreover $w\in M$ is a multigraded element
of $M$ of multidegree $\b$ if $w\in M_\b$, and in this case we
write $\deg w=\b$. Let $M$ be  a multigraded finitely generated
$R$-module. $M$ has a minimal multigraded presentation $ \phi:
R^l\mapright{} R^s\mapright{} M \mapright{} 0$ where for a choice
of multigraded generators for $R^s$ and $R^l$, $\phi$ is
represented by $A_M$, a multigraded  {\it presentation matrix} of
$M$. We note that the data consisting of the multidegrees of the
generators of $R^s$ and $R^l$ and the matrix of coefficients
$(c_{ij})$ describes a {\it monomial} matrix as in \cite{MiSt05}.
In particular whenever $c_{ij}\neq 0$, $\a_{ij}=$ (degree of
column $j$)- ( degree of row $i$). We pay special attention to
this set of equations. Let $A=(c_{ij}x^{\a_{ij}})\in
\mathbb{M}_{s\times l}(R)$, $c_{ij}\in \Bbbk$, $\a_{ij} \in
\mathbb{N}^n$.   Whenever $c_{ij}\neq 0$ we consider the equation
$\g_j-\b_i=\a_{ij}$ with unknowns $\g_j$, $\b_i$. We assemble
these equations  to a system $E_A$ of at most $s\cdot l$ equations
and $s+l$ unknowns. The fact that $A$ is multigraded reflects the
consistency of $E_A$. For any  particular solution
$T=(\g_1,\ldots, \g_l,  \b_1,\ldots, \b_s)$ of $E_A$, $\g_j$ gives
the degree of the $j^{\rm th}$ column of $A$  and $\b_i$ the
degree of the $i^{\rm th}$ row of $A$.  Moreover for any such
solution $T$ we let $F_{1}$, $F_{0}$ be the free multigraded
modules with bases $B_{1}=\{w_j: j\in [l],\ \deg w_j=\g_j\}$  and
$B_{0}=\{ v_i: i\in [s],\ \deg v_i=\b_i\}$ respectively, and let
$\phi_T: F_{1}\mapright{} F_{0}$, $\phi_T(w_j)=\sum_i c_{ij}
x^{\a_{ij}} v_i$. The module $M_T=\coker \phi_T$ is  multigraded
and has $A$ as a  presentation matrix. We will occasionally write
$M_A$ for $M_T$.

The multigraded module $M$ is called {\em squarefree} if the
function $M_\a \rightarrow M_{\a+\b}$: $ y\mapsto x^\b y $ is a
bijection whenever $\supp(\b)\subset \supp(\a)$, see \cite{Ya00}.
In some sense it suffices to study squarefree multigraded modules:
as is shown in \cite{BrHe95} or \cite{ChDe01}, if $M$ is any
multigraded module then there is a a squarefree multigraded module
$L$ with the same homological properties as $M$. In this paper we
show that the multigraded matrix $A$ is the presentation matrix of
a squarefree module if and only if  there exists a squarefree
solution to $E_A$. Such a  matrix is called  {\em squarefree}. It
follows that all nonzero entries of a squarefree matrix have
squarefree degrees.

When $M=R/I$ and $I$ is a monomial squarefree ideal, the
simplicial complex $\D_I=\{ \{i_1,\ldots, i_t\} \subset
\{1,\ldots, n\}:\  x_{i_1}\cdots x_{i_t}\notin I\}$, is well
studied and properties of $M$ translate to combinatorial
properties of $\D_I$.  We generalize the above to any multigraded
squarefree module $M$. For this we use  a sequence of monomial
squarefree ideals that are associated to the presentation of $M$.
When the multigraded presentation matrix $A_M$ is of uniform rank
such a set of ideals is explicitly computed in terms of the rows
of $A_M$. A preliminary version of these results (without proofs)
has appeared in \cite{Ch06}.

We describe the main results in each section.  We show that if the
multigraded squarefree module $M$ has a minimal multigraded
presentation $ \phi: R^l\mapright{} R^s\mapright{}$ $ M
\mapright{} 0$ then there are $s$ squarefree monomial ideals
$I_1,\ldots, I_s$ that determine a multigraded $\Bbbk$-basis $M$,
Theorem \ref{gen_decomp} and Corollary \ref{gen-coeff_basis}. This
translates as follows in Gr\"{o}bner basis language: consider a
term order on $R^s$ based on any monomial order on $R$ and an
ordering of the multigraded bases elements $v_i, i\in [s]$ of
$R^s$; the initial module of the image of $\phi$ is a direct sum
$I_1 v_1\oplus\cdots\oplus I_s v_s$.  It follows that if $A$ is
any multigraded squarefree matrix  then there is a multigraded
squarefree module $M$ with presentation matrix $A$, Corollary
\ref{squarefree_char}.   We study the annihilator ideal of a
multigraded squarefree module $M_A$ when $A$ is an $s\times l$
matrix of uniform rank:  when $s> l$ we show that $\ann(M)=0$
while when $s\le l$ we show that $\ann(M)$ equals the radical of
the ideal generated by the $s\times s$ minors of $A$, Theorem
\ref{annihilator}.

In the next section we study in more detail the case of a
squarefree multigraded module whose presentation matrix $A_M$ is
of uniform rank. In this case we show that the squarefree monomial
ideals that determine a basis of $M$ are  generated by least
common multiples of monomials in the appropriate rows of $A_M$,
Theorem \ref{k-basis-thm}. Their intersection equals $\ann(M)$.
Thus the dimension of $M$ can be computed based on these ideals.

In the last two sections of this paper we assume $M$ to be a
squarefree multigraded module. In section 4, for each $\a\in
\mathbb{Z}^n$ we construct a cochain complex and use it to compute
the $\a$-graded betti numbers of $M$.  In the last section for
each $\a\in \mathbb{Z}^n$ we construct a complex to calculate the
$\a$-graded piece of the local cohomology of $M$.

We refer to \cite{Ei97}, \cite{BrHe98} and \cite{MiSt05} for
undefined terms and notation. We also want to thank the referee
for suggesting the more general version of Theorem
\ref{gen_decomp} and Corollary \ref{gen-coeff_basis} and the
generalization of the last two sections.

\section{Squarefree multigraded matrices}

For $\a=(a_i)\in \mathbb{Z}^n$, we let $\supp(\a)$ $= \{ i:\ a_i
\neq 0\}\subset [n]$. When $\a \in \mathbb{N}^n$ we write $\s_\a$
for $\supp(\a)$ and denote by $q_\a$ the squarefree vector so that
$\s_{\a}=\s_{q_\a}$. If $t\in \mathbb{N}$ by $[t]$ we denote the
set $\{ 1,\ldots, t\}$.

\begin{definition}\label{sq_free_def} A multigraded
$s\times l$ matrix $A=(c_{ij}x^{\a_{ij}})$ is called {\it
squarefree} if   the system $E_A$ has a squarefree solution $T$.
\end{definition}

\begin{remark} Let $M$ be a squarefree multigraded module and let $\phi$ be
a minimal multigraded presentation $\phi$ of  $M$. Since the
kernel of $\phi$ is a squarefree module \cite{Ya00} it follows
that the minimal multigraded generating sets of $M$ and  $\ker
\phi$ have squarefree degrees. Thus $A_M$ is squarefree.
\end{remark}

\noindent  Let $T=(\g_1,\ldots, \g_l,\b_1,\ldots, \b_s)$ be a
squarefree solution of $E_A$, $\phi_T: F_{1}\mapright{} F_{0}$.
For any monomial order on $R$ and an ordering of the basis
elements of $F_{0}$, we let $>$ be the following monomial order on
$F_{0}$: $u v_i>u' v_j$ if $v_i>v_j$ or $i=j$ and $u>u'$. We
denote by $\ini(\ima \phi_{T})$ the initial module of
$\ima(\phi_T)$ with respect to $>$.

\begin{theorem}\label{gen_decomp} Let $A=(c_{ij}x^{\a_{ij}})$ be
a multigraded squarefree $s\times l$ matrix, $T$  a squarefree
solution of $E_A$ and $>$  a term order on $F_{0}$ as above. There
exist squarefree monomial ideals
 $I_1, \ldots, I_s$ of $R$ such that \[ \ini(\ima
\phi_{T})= I_{1} v_1\oplus\cdots\oplus I_{s} v_s
\]
\end{theorem}

\begin{proof}
Without loss of generality we can  assume that $v_s> \ldots
>v_1$.
 Let $f\in \im \phi$, $f$ multigraded,   $\deg f=\a$,
$\ini(f)=x^{\a_i}v_i$. Thus for $j\in [s]$ and $t\in [l]$, there
exist $c_j\in \Bbbk$ and $r_t\in R_{\a-\g_t}$ such that
\[f= \sum_{j\le i} c_j x^{\a_j} v_j=\sum_t r_t \phi_{T}(w_t)\
 .  \] For $t\in [l]$, let $d_t=\deg r_t=\a-\g_t$.
 Since $\g_t\in \{0,1\}^n$, it follows that
whenever $r_t\neq 0$, $d_t-(\a-q_\a)=q_\a-\g_t \in \{0,1\}^n$.
Moreover  since  $\a_j=\a-\b_j$ and $\b_j\in \{0,1\}^n$, it
follows that whenever  $c_j\neq 0$, ${\a_j}'=\a_j-(\a-q_\a) \in
\{0,1\}^n$.
 Thus ${r_t}'=r_t/x^{\a-q_a}\in R$ for $t\in [l]$ and
\[ f'=\sum_t {r_t}' \phi_{T}(w_t) \in \ima (\phi_{T})\ .\] Since
$\ini(f')=  c_i x^{{\a_i}'} v_i$  and $\a_i-{\a_i}' \in
\mathbb{N}^n$ we are done.
\end{proof}

We let   $M_T=\coker {\phi_{T}}$ and write $\overline{g}$ for
$g+\ima \phi_{T}$.   We note the following:

\begin{corollary}\label{gen-coeff_basis}  Let $A=(c_{ij}x^{\a_{ij}})$ be
a multigraded squarefree $s\times l$ matrix, $T$ a squarefree
solution of $E_A$, $M_T=\coker {\phi_{T}}$. There exist simplicial
complexes $\D_1,\ldots, \D_s$ such that
\begin{enumerate}
\item{} the set $B(M_T)=\{ \overline{x^\b v_i}:\ \s_\b\in \D_i, \ i\in
[s]\}$ is a $\Bbbk$-basis of $M_T$,
\item{} if $\a\in
\mathbb{N}^n$ and $\s({\a}) \notin \D_i$, then for each $j< i$
there are unique $r_{i,j,\a}\in \Bbbk$ such that $r_{i,j,\a}=0$
when $\s({{\a}+\b_i-\b_j})\notin \D_j$ and
\[\overline{x^{\a} v_i}=\ds{\sum_{j<i}}\ r_{i,j,\a}\
\overline{x^{{\a}+\b_i-\b_j} v_j}\ .\]
\end{enumerate}
\end{corollary}

\begin{proof} We let $I_1,\ldots, I_s$ be the ideals of
of Theorem \ref{gen_decomp} with  $v_s> \ldots
>v_1$ and we let
$\D_i$ be the simplicial complex $\D_{I_i}$. The first part
follows from Macaulay's Lemma, see for example \cite[Theorem
15.3]{Ei97}. For the second part we note that  if $\s(\a)\notin
\D_i$ and $x^\a \in I_i$ then there is an $f_\a \in \im
(\phi_{T})$ such that $\ini (f_\a)= x^\a v_i$ and
\[ x^\a v_i - f_\a=\sum_{j< i} c_j x^{\a_j} v_j\ . \]
In particular $\a_j=\a+\b_i-\b_j$. A repeated application of this
remark gives the desired result.
\end{proof}

The next corollary justifies the definition of a squarefree
matrix.

\begin{corollary}\label{squarefree_char} Let $A$ be a multigraded squarefree matrix. Then $A$
is the presentation matrix of a multigraded squarefree module $M$.
\end{corollary}

\begin{proof}
Let $T$ be a squarefree solution of $E_A$, $M_T=\coker
{\phi_{T}}$.  By Corollary \ref{gen-coeff_basis}, $(M_{T})_\a\cong
(M_{T})_{q_\a}$ and $M_{T}$ is squarefree.
\end{proof}

The {\it join} of $\a_1,\ldots, \a_t\in \mathbb{N}^n$ denoted
$\join(\a_1,\ldots, \a_t)$ is the vector with components the
maximum of the corresponding components of the $\a_1,\ldots,\a_t$.
We will need the following lemma.

\begin{lemma}\label{lcm_useful} Let $A=(c_{ij}x^{\a_{ij}})$ be
a multigraded $s\times l$ of uniform rank where $\a_{ij}$ are
squarefree, and let $t,q\in [l]$ and $f, i_1, \ldots, i_r \in
[s]$. Then
\[\frac{\lcm( x^{\a_{tf}},x^{\a_{ti_1}},\cdots,x^{\a_{ti_{r}}}) }{x^{\a_{tf}}}=
\frac{\lcm( x^{\a_{qf}},x^{\a_{qi_1}},\cdots,x^{\a_{qi_{r}}})
}{x^{\a_{qf}}}\ , \] or equivalently
\[\join (\a_{tf}, \a_{ti_1},\ldots, \a_{ti_{r}}) -{\a_{tf}}=
\join (\a_{qf},\a_{qi_1},\cdots,\a_{qi_{r}}) -\a_{qf}\ .\]
Moreover if $f, j_1,\ldots, j_r\in [s]$ and $ q,t\in [l]$ then
\[\join (\a_{ft}, \a_{j_1t},\ldots, \a_{j_{r}t}) -{\a_{ft}}=
\join (\a_{fq},\a_{j_1q},\cdots,\a_{j_{r}q}) -\a_{fq}\]
\end{lemma}

\begin{proof} We will show the first equality. We note that the
last equality is a consequence of the first, since $A^T$ is
multigraded. The expressions on either side of the equation are
squarefree monomials. Suppose that the variable $x_j$ divides the
left hand side expression. This implies that $x_j$ does not divide
 $x^{\a_{tf}}$ and  $x_j$ divides $x^{\a_{ti_h}}$ for some $i_h$, where $h=1,\ldots,r$.
 Since $c_{tf}c_{qi_h} x^{\a_{tf}+\a_{qi_h}}-
c_{ti_h}c_{qf}x^{\a_{ti_h}+\a_{qf}}$ is a minor of $A$ and $A$ is
multigraded it follows that $x_j$ divides $x^{\a_{qi_h}}$ and
$x_j$ does not divide $x^{\a_{qf}}$. Thus $x_j$ divides the right
hand side of the equation.
\end{proof}

We can now prove the following:

\begin{proposition} \label{repr_mult_matrix}
Any  multigraded matrix of uniform rank whose entries have
squarefree degrees is squarefree.
\end{proposition}

\begin{proof} Let $A=(c_{ij}x^{\a_{ij}})$
be a multigraded  of uniform rank, where $\a_{ij}\in \{0,1\}^n$
for $i\in [s]$, $j\in [l]$. By
 Lemma \ref{lcm_useful}  a solution to
$E_A$ is given by $\g_j=\join(\a_{ij}:\ i \in[s])$ for $j\in [l]$,
and $\b_i=\g_1-\a_{i1}$ for $i\in [s]$.
\end{proof}

\begin{remarks} $ $
\begin{itemize}
\item{} We note that when $A$ is of uniform rank then $E_A$ has
one degree of freedom.
\item{} When $A=(c_{ij}x^{\a_{ij}})$ is not of uniform rank then $E_A$ might not
have squarefree solutions even when $\a_{ij}\in \{0,1\}^n$ for
$c_{ij}\neq 0$. For example, let $R=\Bbbk[x,y]$ and
\[A=\begin{bmatrix} x& y\cr
                    0& x\end{bmatrix}.
                    \]
The general solution of $E_A$ consists of $\g_1=(2+t,s)$,
$\g_2=(1+t,1+s)$, $\b_1=(1+t,s)$, $\b_2=(t,s+1)$.
\item{} When $A=(c_{ij}x^{\a_{ij}})$ is of uniform rank and $s=1$ then $B_{0}=\{v_1: \b_1=\deg v_1=0\} $,
$B_{1}=\{w_j: \g_j=\deg w_j=\a_{1j} \}$ and $M_T=R/I$ where
$I=\langle x^{a_{11}},\ldots, x^{a_{1l}}\rangle$.
\end{itemize}
\end{remarks}

We will close this section by examining the annihilator of a
multigraded squarefree module $M$ in terms of its multigraded
presentation matrix $A$. It is easy to see that $\ann(M)$ is
generated by monomials. If $A$ is $s\times l$ and $\l\ge s$ we
denote by $\Fitt_0(M)$ the ideal of the $s\times s$ minors of $A$.
We note that the generators of $\sqrt{\Fitt_0(M)}$ are least
common multiples of the entries in a diagonal of $A$ of length
$s$. It is well known, see \cite{Ei97}, that $\Fitt_0 (M)\subset
\ann(M)$ and that $\ann(M)^s \subset \Fitt_0(M)$. In the case
where  $\ann(M)$ is a monomial ideal it follows that
$\ann(M)\subseteq \sqrt{\Fitt_0(M)}$. For what follows we write
$\diag(x^\a,s)$ for the  $s\times s$ identity  matrix times
$x^\a$. We will use the following lemma which characterizes the
elements of $\ann(M)$.

\begin{lemma}\label{ann} Let $A$ be multigraded squarefree
$s\times l$ matrix and let $M$ be a module with presentation
matrix $A$. The annihilator of $M$ consists of all monomials
$x^\a$ such that the linear system
\[ A X=\diag(x^\a,s)\]
has a solution in $\mathbb{M}_{l\times s}(R)$.
\end{lemma}

\begin{proof} Let $\phi: F_{1}\mapright{} F_{0}\mapright{}M\mapright{}0$
be such that  for a basis $\{v_i:\ i\in [s]\}$ of $F_0$ and a
basis $\{ w_j: \ j\in [l]\}$ of $F_1$, $\phi$ is represented by
$A=(c_{ij}x^{\a_{ij}})$. Since $M=F_{0}/\im (\phi)$, it follows
that $x^\a\in \ann(M)$ if and only if $x^\a v_i\in \im(\phi_{})$.
Thus $x^\a\in \ann(M)$ if and only if for any $i\in [s]$ there
exist $r_{i1},\ldots, r_{il}$ such that
\begin{equation}\label{system_ann} x^\a v_i=\ds{\sum_{j=1}^l r_{ij}
\sum_{t=1}^s } (c_{tj}x^\a_{tj}) v_t.\end{equation} We let $C_i$
be the $i$th column of $\diag(x^\a,s)$. Since $\{v_i:\  i\in
[s]\}$ is a basis for $F_{0}$ the system \ref{system_ann} is
consistent if and only if $AX=C_i$ is consistent for each $i\in
[s]$.
\end{proof}

 If $K_1\subset [s]$ and $K_2\subset [l]$ we denote by
$A[K_2,K_1]$ the submatrix of $A$ consisting of the rows indexed
by $K_1$ and columns indexed by  $K_2$.
 The next proposition computes the annihilator of the squarefree
 module $M$ when the presentation matrix is of uniform rank.

\begin{proposition}\label{annihilator}  Let $A$ be
a multigraded squarefree $s\times l$ matrix of uniform rank.
Suppose that $A$ is the presentation matrix of $M$. If $l<s$ then
$\ann(M)=0$. Otherwise $\ann(M)=\sqrt {\Fitt_0(M)}$.
\end{proposition}

\begin{proof}  If $l< s$ and $x^\a\in \ann(M)$ then by Lemma
\ref{ann} assume that $Z$
is such that
\[ A Z=\diag(x^\a,s)\ .\]
 Without loss of generality we can
assume that the first column $Z_1$ of $Z$ is nonzero. It follows
that $A[[l], \{2,\ldots, l+1\}]\cdot Z_1=0$ and thus  $\det A[[l],
\{2,\ldots, l+1\}]=0$, a contradiction since $A$ is of uniform
rank.

Suppose now that $l\ge s$. We will show that $\sqrt
{\Fitt_0(M)}\subset \ann(M)$. Let $I=\sqrt{\Fitt_0(M)}$ and let
$x^{\a}$ be a minimal generator of $I$.  It follows that there is
a set $K=\{j_1,\ldots, j_s:\ j_1<j_2<\ldots <j_s\}\subset [l]$
such that $x^{\a}=\det A[K,[s]]$. Since $A$ is multigraded it
follows that for any $i\in [s]$,  $\det
A[K,[s]]=c_{K,t}x^{\a_{ij_t}} \det A[K\setminus j_t, \{ 1,\ldots,
\hat{i},\ldots,s\}]$, where $c_{K,t}\in \Bbbk$.  For $i\in [s]$,
we let $C_i$ be as in the proof of Lemma \ref{ann}. By Cramer's
rule it follows that  the system $ A[K,[s]] X=C_i $ has a solution
$Z=(z_{j1})\in \mathbb{M}_{s\times 1}(R)$ with entries in $R$. We
extend $Z$ to a solution $Y\in \mathbb{M}_{l\times 1}(R)$ for the
system $ A X=C_i$ by setting $y_{j1}= z_{j1}$ when $ j\in K$ and
letting $y_{j1}= 0$ if $j\notin K$. Thus by Lemma \ref{ann},
$x^{\a}\in \ann(M)$ and $I\subset \ann(M)$ as desired.
\end{proof}

\section{Squarefree matrices of uniform rank}

In the previous section for any multigraded squarefree matrix $A$
and any solution $T$ of $E_A$ we proved the existence of a
sequence of squarefree monomial ideals that provide a basis for
$M_T$. In this section we  compute these ideals when $A$ is
squarefree of uniform rank.

\begin{definition}\label{basicdef} Let $A=(c_{ij}x^{\a_{ij}})$ be a multigraded
squarefree matrix of uniform rank with $l\ge s$. For $i\in [s]$ we
let
\[I_i:=\langle\  \lcm(x^{\a_{ij_1}}, \ldots,x^{\a_{i
j_{s-i+1}}}) :\ 1\le j_1<\cdots < j_{s-i+1}\le l \ \rangle \ .\]
We denote by $\D_i$ the simplicial complex $\D_{I_i}$.
\end{definition}

For the rest of this section we assume that $A$ is of uniform rank
as above.

\begin{example}\label{basic_example}
Let $R=\Bbbk[x,y,z,w]$ and
\[A=\begin{bmatrix} xy& xz\cr
                    wy& 2wz\end{bmatrix}.
                    \]
$A$ is the matrix of the $R$-module homomorphism $\phi_1:
F_1\mapright{} F_0$ where $F_1=Rw_1\oplus Rw_2$,
 $F_0= R v_1\oplus Rv_2$,  $\deg w_1=\g_1=(1,1,0,1)$, $\deg
w_2=\g_2=(1,0,1,1)$, $\deg v_1=\b_1=(0,0,0,1)$,  $\deg
v_2=\b_2=(1,0,0,0)$, $M_A=\coker \phi_1$. Here $I_1=(xyz)$ and
$I_2=(wy, wz)$. Below we graph the simplicial complexes $\D_1$,
$\D_2$.

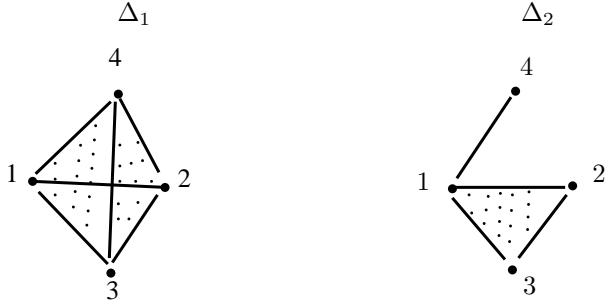
\scalebox{1} 
{
\begin{pspicture}(0,-2.1514063)(8.559688,2.1514063)
\psdots[dotsize=0.12](0.2778125,-0.5470312)
\psdots[dotsize=0.12](1.4178125,0.61296874)
\psdots[dotsize=0.12](2.0378125,-0.62703127)
\psdots[dotsize=0.12](1.3178124,-1.7670312)
\psline[linewidth=0.04cm](0.2778125,-0.5470312)(2.0378125,-0.62703127)
\psdots[dotsize=0.12](6.6978126,0.65296876)
\psdots[dotsize=0.12](5.8578124,-0.64703125)
\psdots[dotsize=0.12](7.4578123,-0.6070312)
\psdots[dotsize=0.12](6.6578126,-1.7270312)
\psline[linewidth=0.04cm](1.3178124,0.5129688)(0.3178125,-0.44703126)
\psline[linewidth=0.04cm](0.3578125,-0.68703127)(1.2778125,-1.6470313)
\psline[linewidth=0.04cm](1.9378124,-0.72703123)(1.3378125,-1.6270312)
\psline[linewidth=0.04cm](1.4378124,0.55296874)(1.9578125,-0.42703125)
\psline[linewidth=0.04cm](1.3778125,0.5129688)(1.2978125,-1.5670313)
\psline[linewidth=0.04cm](6.6378126,0.5729687)(5.9178123,-0.50703126)
\psline[linewidth=0.04cm](5.8978124,-0.62703127)(7.3578124,-0.62703127)
\psline[linewidth=0.04cm](7.3778124,-0.72703123)(6.7378125,-1.5670313)
\psline[linewidth=0.04cm](5.8578124,-0.76703125)(6.5578127,-1.5870312)
\psline[linewidth=0.04cm,linestyle=dotted,dotsep=0.16cm](1.1178125,0.21296875)(0.9978125,-1.1470313)
\psline[linewidth=0.04cm,linestyle=dotted,dotsep=0.16cm](0.9178125,-0.04703125)(0.8178125,-0.90703124)
\psline[linewidth=0.04cm,linestyle=dotted,dotsep=0.16cm](0.5778125,-0.28703126)(0.5778125,-0.7470313)
\psline[linewidth=0.04cm,linestyle=dotted,dotsep=0.16cm](1.4178125,-0.06703125)(1.6978126,-0.02703125)
\psline[linewidth=0.04cm,linestyle=dotted,dotsep=0.16cm](1.4178125,-0.34703124)(1.7578125,-0.26703125)
\psline[linewidth=0.04cm,linestyle=dotted,dotsep=0.16cm](1.4178125,-0.5470312)(1.8778125,-0.5470312)
\psline[linewidth=0.04cm,linestyle=dotted,dotsep=0.16cm](1.3978125,-0.84703124)(1.6978126,-0.82703125)
\psline[linewidth=0.04cm,linestyle=dotted,dotsep=0.16cm](1.3978125,-1.0470313)(1.5778126,-1.0470313)
\psline[linewidth=0.04cm,linestyle=dotted,dotsep=0.16cm](6.0778127,-0.70703125)(6.1178126,-0.9870312)
\psline[linewidth=0.04cm,linestyle=dotted,dotsep=0.16cm](6.3378124,-0.72703123)(6.2778125,-1.1270312)
\psline[linewidth=0.04cm,linestyle=dotted,dotsep=0.16cm](6.4978123,-0.68703127)(6.4778123,-1.2870313)
\psline[linewidth=0.04cm,linestyle=dotted,dotsep=0.16cm](6.6578126,-0.70703125)(6.6178126,-1.3670312)
\psline[linewidth=0.04cm,linestyle=dotted,dotsep=0.16cm](6.8778124,-0.6670312)(6.8578124,-1.2070312)
\usefont{T1}{ptm}{m}{n}
\rput(1.37875,1.1029687){4}
\usefont{T1}{ptm}{m}{n}
\rput(6.849844,0.98296875){$4$}
\usefont{T1}{ptm}{m}{n}
\rput(0.0046875,-0.43703124){1}
\usefont{T1}{ptm}{m}{n}
\rput(5.4715624,-0.5570313){$1$}
\usefont{T1}{ptm}{m}{n}
\rput(1.3454688,-1.9970312){3}
\usefont{T1}{ptm}{m}{n}
\rput(6.8853126,-1.9370313){$3$}
\usefont{T1}{ptm}{m}{n}
\rput(2.2964063,-0.49703124){2}
\usefont{T1}{ptm}{m}{n}
\rput(7.80875,-0.43703124){$2$}
\usefont{T1}{ptm}{m}{n}
\rput(1.6292187,1.6829687){$\Delta_1$}
\usefont{T1}{ptm}{m}{n}
\rput(7.0,1.6829687){$\Delta_2$}
\end{pspicture}
}

\end{example}
We note that when $s=1$ then  $I_1=\langle x^{a_{11}},\ldots,
x^{a_{1l}}\rangle$ and the unique multigraded  $\Bbbk$-basis of
$R/I_1$ is the set $\{ x^{\b}: \s_\b \notin \D_{1}\}$. We recall
that if
 $T=(\g_j: j\in
[l], \b_i: i\in [s])$ is a solution of $E_A$ then  $F_{1}$,
$F_{0}$ are the free multigraded modules with bases $B_{1}=\{w_j:
j\in [l],\ \deg w_j=\g_j\}$, $B_{0}=\{ v_i: i\in [s],\ \deg
v_i=\b_i\}$ respectively, $\phi_T: F_{1}\mapright{} F_{0}$,
$\phi_T(w_j)=\sum_i c_{ij} x^{\a_{ij}} v_i$ and $M_A=\coker
\phi_T$.

 \begin{definition}\label{k-basis-def}    We  define
\[\mathcal{B}_{jA}:=\{ \overline{x^\b v_j}: \ x^\b\notin I_j\}=
\{ \overline{x^\b v_j}: \ \s_{\b}\in \D_j\} \ ,\]
\[ \mathcal{B}_A:= \bigcup_{j=1}^s \mathcal{B}_{jA}\]
\end{definition}

The elements of $\mathcal{B}_A$ are multigraded and $\deg
\overline{x^\b v_j}=\b+\b_j$.

\begin{theorem}\label{k-basis-thm} Let $A=(c_{ij}x^{\a_{ij}})$ be an
$s\times l$
 multigraded squarefree matrix of uniform rank. The set $ \mathcal{B}_A$
 is a multigraded
$\Bbbk$-basis for $M_A$.
\end{theorem}

\begin{proof}
We use induction on $s$. For $s=1$ the theorem is clear. Let
$s>1$. First we  show that the elements of $\mathcal{B}_A$ are
linearly independent. Suppose that
\[ \sum_{j\in [s]} \sum_{i\in K_j} k_{ji} x^{\b_{ji}}v_j=\ds{\sum_{f=1}^l \sum_{t=1}^s }
r_{f}(c_{tf}x^{\a_{tf}}) v_t\] is a dependence relation;  $K_j$ is
a finite index set for each $j\in [s]$, $k_{ji}\in \Bbbk$, and
$\overline{x^{\b_{ji}} v_j} \in \mathcal{B}_{jA}$. It follows
that
\begin{equation}\label{sum_induction_start}
\sum_{i\in K_1} k_{1i} x^{\b_{1i}} v_1=\ds{\sum_{f=1}^l} r_{f}
c_{1f}x^{\a_{1f}} v_1
\end{equation} while
\begin{equation}\label{induction_submatrix}
\sum_{ t=2}^s \sum_{i\in K_t} k_{ti} x^{\b_{ji}}
v_t=\ds{\sum_{f=1}^l \sum_{t=2}^s } r_{f} c_{tf}x^{\a_{tf}}
v_t.\end{equation}

\noindent Let ${F_0}'=Rv_2\oplus\cdots\oplus Rv_s$, $\phi':
F_{1}\mapright{} {F_0}'$ be the $R$-homomorphism whose matrix with
respect to the bases $\{w_i: i\in[l]\}$ of $F_{1}$ and $\{v_i:
i\in \{2,\ldots,s\}\}$ of ${F_0}'$ is the submatrix $A'=A[
\{1,\ldots,l\}, \{2,\ldots,s\}]$ of $A$. Thus $M_{A'}={F_{0}}'/\im
\phi'$. We note that for $j=1,\ldots, s-1$ there is a one-to-one
correspondence between the elements of $\mathcal{B}_{jA'}$ and
$\mathcal{B}_{j+1A}$. Thus the expression of Equation
\ref{induction_submatrix} translates to a dependence relation for
the elements $\mathcal{B}_{A'}$. According to the induction
hypothesis this implies that $k_{ti}=0$ for $t\ge 2$ and $i\in
K_t$. Thus it remains to show that $k_{1i}=0$. We examine the
coefficients $r_f$ that appear on the right hand side of
Equations \ref{sum_induction_start} and \ref{induction_submatrix}.
Since the sum on the left hand side of
 Equation \ref{induction_submatrix} is zero, it follows that
$r_1w_1+\cdots+r_l w_l\in \ker \phi_1'$. According to the
description of the free resolution of $M_{A'}$ see \cite{ChTc03}
  it
follows that for $1\le j\le l$,
\[ r_j\in \langle\ \frac{\lcm(
x^{\a_{2j}},x^{\a_{2i_1}},\cdots,x^{\a_{2i_{s-1}}})
}{x^{\a_{2j}}}: \ 1\le i_1<\cdots<i_{s-1}\le l, i_t \neq j\
\rangle.\] By Lemma \ref{lcm_useful},
\[\frac{\lcm( x^{\a_{2j}},x^{\a_{2i_1}},\cdots,x^{\a_{2i_{s-1}}}) }{x^{\a_{2j}}}=
\frac{\lcm( x^{\a_{1j}},x^{\a_{1i_1}},\cdots,x^{\a_{1i_{s-1}}})
}{x^{\a_{1j}}}\] and $r_j x^{\a_{1j}}\in I_1$, a contradiction.

Next we show that $\mathcal{B}_A$ generates $M_A$. More precisely
we will show that if $x^\a\in I_i$ then $\overline{x^\a v_i}$ can
be written as a $\Bbbk$-linear combination of elements of
$\mathcal{B}_{1A}\cup\ldots\cup \mathcal{B}_{i-1A}$ of degree
$\a+\b_i$. We first show this for the elements of $I_1$. Let
$x^\a$ be the least common multiple of the entries in the first
row corresponding to the columns indexed by the set $K\subset
[l]$, where $|K|=s$. Since $A$ is multigraded $\det A[K, [s]]$
divides $x^\a \det A[K\setminus f,\{ 2,\ldots,s\} ]$ for any $f\in
K$. By Cramer's rule it follows that there is a matrix $Z=(z_i)\in
\mathbb{M}_{s\times 1}(R)$ such that
\[ A[K, [s]] Z=\begin{bmatrix}x^\a&
0&\cdots& 0\end{bmatrix}^T\ .\] It follows that if
$K=\{i_1,\ldots, i_s\}$ then $x^\a v_1=\phi_1(z_1 w_{i_1}+\cdots+
z_s w_{i_s})$. Therefore $\overline{x^\a v_1}=0$. We now assume
that the statement holds for $j<t$. Let $K=\{ i_1,\ldots,
i_{s-t+1}\}$ and $x^\a\in I_t$ be equal to $\lcm(x^{a_{ti_f}}:
i_f\in K)$.  Let $Z=(z_f)\in \mathbb{M}_{s-t+1\times 1}(R)$ be
such that
\[ A[K, \{t,\ldots, s\}] Z
=\begin{bmatrix}x^\a& 0&\cdots& 0\end{bmatrix}^T.\] For $w= z_1
w_{i_1}+\cdots+ z_{s-t+1} w_{i_{s-t+1}}$ we have:
\[\phi_1(w)= z_1(\sum_{j<t} c_{j i_1} x^{\a_{j i_1}} v_j)+\cdots+
           z_{s-t+1}(\sum_{j<t} c_{j i_{s-t+1}} x^{\a_{j i_{s-t+1}}}
           v_j)+ x^\a v_t\]
     \[      = \sum_{j<t} (\sum_{r=1}^{r=s-t+1} z_r c_{j i_r} x^{\a_{j i_r}}
     v_j)+x^\a v_t.\]
     Therefore
     \[ \overline{x^\a v_t}=-\sum_{j<t} (\sum_{r=1}^{r=s-t+1}c_{j i_r} \overline{z_r  x^{\a_{j i_r}}
     v_j})\]
and we are done by the induction hypothesis as applied to each of
the summands $\overline{z_r  x^{\a_{j i_r}}
     v_j}$.
\end{proof}

The next corollary is immediate and we omit its proof.

\begin{corollary} Let $A$ be an
$s\times l$
 multigraded squarefree matrix of uniform rank,
$>$ be a term order on $F_0$ based on a monomial order of $R$ with
$v_s> \cdots >v_1$. Then
\[ \ini(\im (\phi_{1}))=
I_{1}v_1\oplus\cdots\oplus I_{s} v_s\ .
\]
\end{corollary}

\begin{example} \label{coeff} Let $A$ be the matrix of  Example
\ref{basic_example} and $\a=(0,1,0,1)$.  Then  $\overline{wy\
v_2}=-\overline{xy\ v_1}$ and $r_{2, 1, \a}= -1$. When $\a=
(0,0,1,1)$ then  $\overline{wz\ v_2}=-1/2\overline{ xz\ v_1}$ and
$r_{2, 1,\a}= -1/2$. Note that $xyz v_1=\phi_1(2z\ w_1-y \ w_2)$
and $\overline{xyz v_1}=0$.
\end{example}

A different ordering of the basis elements $v_i$ one would modify
Definition \ref{basicdef} to get  a different set of ideals and a
potentially different $\Bbbk$-multigraded basis of $M_A$. For
example if $v_1>\cdots>v_s$ then   the $i^{\rm th}$ ideal should
be generated by all least common multiples of $i$ monomials in the
$i^{\rm th}$ row of $A$. Next we describe the annihilator of $M$
in terms of the ideals $I_i$.

\begin{proposition} \label{annihilator_as_intersection} Let $A=(c_{ij}x^{\a_{ij}})$ be an $s\times l$
 multigraded squarefree matrix of uniform rank, $l\ge s$. Then
\[\ann(M_A)=I_1\cap \cdots\cap I_s\ . \]
\end{proposition}
\begin{proof} We use Proposition \ref{annihilator}.
We first show that the intersection $I_1\cap \cdots\cap I_s$ is
contained in $\ann(M_A)$. Let $x^{\a_i}\in I_i$ for each $i\in
[s]$: $x^{\a_i}$ determines a (not necessarily unique) subset
$L_i\subset[l]$ of cardinality $s-i+1$ such that
$x^{\a_i}=\lcm(x^{\a_{it}}:\  t\in L_i)$. It follows that there is
a set $K=\{i_1,\ldots, i_s\}\subset [l]$ of cardinality $s$ such
that $i_t\in L_t$.  Since $A$ is multigraded it follows  that
$\det A[K, [s]]$ divides $\lcm(\a_1,\ldots,\a_s)$.

For the reverse containment, let $x^{q}$ be a generator of
$\ann(M_A)$: thus there is an ordered set $K=\{j_1,\ldots,j_s\}$
such that $cx^\a=\det A[ K, [s]]$ and $q=q_\a$.  Let
$x^{\a_i}=\lcm( x^{\a_{ij_t}}:\ t=i,\ldots, s)$. It is clear that
$x^{\a_i} \in I_i$. Since $A$ is multigraded it follows that
$x^{q_\a}=\lcm ( x^{\a_1},\ldots, x^{\a_s})$.
\end{proof}

The following is now immediate:

\begin{corollary} Let $A$ be as above. The dimension of $M_A$ is equal to the least of the
codimensions of the ideals $I_j$.
\end{corollary}

\section{Betti numbers of $M$}

Let $A=(c_{ij}x^{\a_{ij}})$ be an $s\times l$
 multigraded squarefree matrix, $T$ a squarefree solution of
 $E_A$,
 $M_A=\coker  {\phi_{T}}$. We consider the ideals $I_1,\ldots, I_s$ of
 Theorem \ref{gen_decomp} with respect to a term order induced by
 $v_s>\cdots>v_1$. We start by
assembling some notation.

\begin{notation} $ $

\begin{itemize}
\item{} For $i\in [s]$, we let $\b_i=\deg(v_i)$. For $\a \in \mathbb{Z}^n$ we let
$\a_j=\a-\b_j$.
\item{} Let $L\subset [n]$. If $t\in L$ then we write $L{\hat t}$ for
the set $L\setminus \{t\}$. If $\s\subset L$ we write $L{\hat \s}$
for the set $L\setminus \s$.
\item{} By $Lt$ we
denote the set $L\cup \{t\}$, by $L\s$  the set $L\cup \s$.
\item{} Let $L\subset[n]$. Then  $\underline{L}=(d_i)$ where $d_i=1$ if $i\in
L$ and $0$ otherwise.
\item{} Let $L=\{ i_1,\ldots, i_t\}$ where $1\le
i_1<\ldots<i_t\le n$. For $r\in [t]$ we let $\sgn[i_r,
L]=(-1)^{r+1}$. For $W\subset L$ we let
\[\sgn[W,L]:=\prod_{w\in W} \sgn[w, L].\]
\item{} Let $(K_\bullet,\th_\bullet)$ be the Koszul complex on the
variables $x_1,\ldots, x_n$. We denote the multigraded generators
of $K_j$  by $e_L$ where $L=\{ i_1,\ldots, i_j\}$ and $1\le
i_1<\ldots<i_j\le n$ and let $\deg e_L=\underline{L}$. Here
\[\th(e_L)=\sum_{t\in L} \sgn[t, L] x_t e_{L{\hat t}}\ .\]
\item{} Let $\D$ be a simplicial complex, $\t\in \D$ and $V$ the vertex set of $\D$.
We partition $V\setminus \t$ into two sets: $V_{\t,\D, 1}=\{
t\notin \t:\ \t t\in \D\}$(=$\link_\D \t$) and $V_{\t,\D, 2}=\{
t\notin \t:\ \t t\notin \D\}$.
\item{} Let $\D$ be a simplicial complex. We let $C^{j}(\D)$ be
the $\Bbbk$-vector space with bases elements $\t^*$ where $\t\in
\D$ and $|\t|=j+1$. We let $(C^\bullet(\D), d)$ be the augmented
cochain complex
\[ C^\bullet(\D):\quad 0\mapright{} C^{-1}(\D) \mapright{} C^{0}(\D) \mapright{d^0}
\cdots \mapright{} C^{n-1}(\D)\mapright{} 0
\] where
\[ d^j(\t^*)=\sum_{t\in V_{\t,\D,1}}\sgn[t, \t t ]\ (\t t)^* .\]
We let $\widetilde{H}^i (\D)=H^i(C^\bullet (\D)$.
\item{}For $\a\in \mathbb{Z}^n$, we write $\a=\a^+-\a^-$ where $\a^+,
\a^-\in \mathbb{N}^n$ and $\supp(\a^+)\cap \supp(\a^-)=\emptyset$.
\item{} Let $\D$ be a simplicial complex and $\a\in \mathbb{N}^n$.
We let \[\D_\a=  \{\  \s\subset \s_{\a}:\ \s\cup \s_{\a-q_\a}\in
\D\ \}\ .\] If $\a\in \mathbb{Z}^n\setminus \mathbb{N}^n$ we let
$\D_\a=\{\}$.
\item{} We let $\D_{j,\a}(A)$ or $\D_{j,\a}$ for short to
be the simplicial complex $(\D_{I_j})_{\a_j}$. Thus
\[\D_{j,\a}(A)=\{\
\s\subset \s_{\a_j}:\ \s\cup \s_{\a_j-q_{\a_j}}\in \D_{I_j} \}\
.\]
\item{} We let
\[(C^\bullet (j,\a),d_j)=(C^\bullet(\D_{j,\a}),d_j)\ .\]
\item{} Let $\t\subset [n]$ be such that $\s_{\b_i}\subset \t\s_{\b_j}$. We define
$f(\t, j, i)=\s_{\b_j} \t \widehat{\s_{\b_i}}$.
\item{} Let $\t^* \in C^r(j,\a)$. Let $w\in [n]$ be such that
$\t w\notin \D_{j,\a}$. The coefficient $ r_{j,i,
\underline{w\t}+\a_j-q_{\a_j}}$ is determined by Corollary
\ref{gen-coeff_basis}. Whenever $ r_{j,i,
\underline{w\t}+\a_j-q_{\a_j}} \neq 0$ it follows that
$\s_{\b_i}\subset \t\s_{\b_j}$ and $f(\t, j, i) \in \D_{i,\a}$ so
that
\[\chi_j(\t^*,w)=\sgn[w, \t w ]\ r_{j,i,
\underline{w\t}+\a_j-q_{\a_j}} \ \frac{\sgn[\t w ,\s_{\a_j}]}{
\sgn[f(\t w, j, i),\s_{\a_i}]} \ (f(\t w, j, i))^*\] is an an
element of
\[\sum_{i<j} C^{r+1+(|\s_{\a_i}|-|\s_{a_j}|)}(i,\a).\]
We let
\[\chi_j(\t^*)= \sum_{w\in V_{\t,
\D_{j,\a},2}}\ \chi_j(\t^*,w)\ .\]
\newline
\end{itemize}

\end{notation}

\begin{example} Let $A$ be the matrix of Example
\ref{basic_example}. Let $\a=(1,0,1,1)$. Then
 $\a_1=(1,0,1,0)$,
$\a_2=(0,0,1,1)$, $\s_{\a_1}=\{1,3\}$,  $\s_{\a_2}=\{3,4\}$ while
$\s_{\b_1}=\{4\}$ and $\s_{\b_2}=\{1\}$. $\D_{1,\a}$ is the line
segment between the vertices $1$ and $3$ while $\D_{2,\a}$
consists of the points $3$ and $4$. It follows that
$f(\{3,4\},2,1)=\{1,3\}$, an element of $\D_{1,\a}$ and $r_{2,1,
\{3,4\}}=-1/2$. Thus $\chi_2(\{3\}^*)=-{\frac{1}{2}} \{1,3\}^*$.
\end{example}

\noindent Next we turn our attention to the minimal multigraded
free resolutions of $M_A$. Let $\a\in \mathbb{Z}^n $ and
$b_{i,\a}(M_A)$ be the $\a$-graded $i$-betti number of $M_A$:
\[b_{i,\a}(M_A)=\dim_{\Bbbk} \Tor_i(M_A,k)_\a=\dim_\Bbbk H_i(M_A\otimes
K_\bullet)_\a=\dim_\Bbbk F_i\otimes \Bbbk\] where $
F_\bullet:\quad 0\mapright{} F_p\mapright{}\cdots
\cdots\mapright{}F_1\mapright{\phi_1} F_0\mapright{}
M_A\mapright{} 0\ $ is a  minimal multigraded free resolution of
$M_A$, ($\phi_1=\phi_T$). It is well known that when $I$ is a
squarefree ideal then
$b_{i,\a}(R/I)=\widetilde{H}^{|\s_\a|-i-1}(C^\bullet((\D_I)_\a))
$,  see \cite{Ho77} or \cite{MiSt05} for a proof. More precisely
there is an isomorphism of complexes:
\begin{equation}\label{graded_cohomology} (R/I\otimes
K_\bullet)_\a\cong C^\bullet((\D_I)_\a),\end{equation} such that
 \[ (R/I\otimes
K_i)_\a\cong C^{|\s_\a|-i-1}((\D_I)_\a). \] We generalize the
isomorphism (\ref{graded_cohomology}) for $M_A$. We will need the
following remark on the signs, proved essentially in \cite{Ho77}.

\begin{remark}\label{signs} Let $\ro \subset \s$, $\t=\s\setminus \ro$ and $t\in \t$. Then
\[ \sgn[t,\t] \sgn[\ro t , \s]=\sgn[ t, \ro t] \sgn[\ro, \s]\ .\]
\end{remark}
We combine the cochain complexes of $\D_{j,\a}$ to construct a new
complex.

\begin{construction} \label{def_cochain_complex} Let
$l_j=|\s_{\a_j}|-|\s_{\a_1}|$. We define
\[ C^t(A,\a):=\sum_{j=1}^s C^{t+l_j}( j, \a)\ ,\]   and let
$\d^t:\ C^t(A,\a)\longrightarrow C^{t+1}(A,\a)$ be such that for
$\t^* \in C^{t+l_j}( j, \a)$, \[\d^{t}(\t^*):=d_j^t(\t^*)
+\chi_j(\t^*)\ . \]
\end{construction}

\begin{theorem} Let $A$ be a squarefree multigraded matrix.
$(C^\bullet(A,\a),\d^\bullet)$ is a cochain complex and there is
an isomorphism of complexes $(C^\bullet(A,\a),\d^\bullet)\cong
(A\otimes K_\bullet)_\a$.
\end{theorem}

\begin{proof} First we note that $M_A\otimes K_\bullet$ is multigraded.
By Corollary \ref{gen-coeff_basis}, a multigraded basis for the
vector space $(M_A\otimes K_\bullet)_\a$ is
\[ \bigcup_j\  \bigcup_{\begin{array}{c} L\subset [n]\cr \a_j-\deg L\in
\mathbb{N}^n\end{array}} \{ \overline{x^{\a_j-\deg L}\ v_j}\otimes
\ e_L:\ \overline{x^{\a_j-\deg L}\ v_j}\ \in B(M_A)\}\ \]
\[=\bigcup_{\begin{array}{c} L\cr \a_j-\deg L\in \mathbb{N}^n\end{array}}
\bigcup_j\ \{ \overline{x^{\a_j-q_{\a_j}}x^{q_{\a_j}-\deg L}
v_j}\otimes\ e_L:\ \s_{\a_j}\hat{ L}\in \D_{j,\a} \}\ . \] To each
element $\overline{x^{\a_j-q_{\a_j}}x^{q_{\a_j}-\deg L}
v_j}\otimes\ e_L$ of this basis we correspond the element
$\sgn[\s_{\a_i}\setminus L, \s_{\a_i}]\ (\s_{\a_i}\setminus L)^*$
of $C^r (j,\a)$ where $r=|\s_{\a_j}|-|L|-1$. Since $(M_A\otimes
K_\bullet )_\a$ is a complex to prove our claim it suffices to
show that the following diagram commutes:
\[ \begin{array}{ccc} \cr C^{t} (A,\a) &\mapright{\d^t} &C^{t+1} (A,\a) \cr
\mapdown{} &
 &\mapdown{} \cr (M_A\otimes K_{|\s_{\a_1}|-t-1} )_\a
 &\mapright{1_{M_A}\otimes\th} &(M_A \otimes
K_{|\s_{a_1}|-t})_\a.\end{array}\] This is a routine check, using
Remark \ref{signs}.
\end{proof}

The following is now immediate and generalizes the well known
formula of the cyclic case.

\begin{corollary}\label{def_cochain_betti} Let $A$ be as
above. Then
\[b_{i,\a} (M_A)=H^{|\s_{\a_1}|-i-1}(C^\bullet(A,\a)).\]
\end{corollary}

When $E^\bullet$ is a complex, by $E^\bullet[-1]$ we mean the
complex $E^\bullet$ pushed in homological degree $-1$: $E^r[-1]:=
C^{r-1}$. This way we can think of $C^\bullet(A,\a)$ as the
cochain complex that results by a succession of mapping cones
$\mathbb{M}(f_{i})$. We start with
$\mathbb{M}(f_{1})=C^\bullet(1,\a)[-1]$ and once
$\mathbb{M}(f_{i-1})$ has been constructed then $\mathbb{M}(f_i)$
is the cokernel of
\[f_i: (C^\bullet(i,\a)[-|\s_{\a_i}|+|\s_{a_1}|],-d)\mapright{d_i'}
\mathbb{M}(f_{i-1})[-1].\]  We note that if $\a\in \mathbf{N}^n$
is not squarefree then for each $i$, $\D_{i,\a}$ is a cone and the
cohomology of $C^\bullet(i,\a)$ is everywhere zero. It follows
that the minimal resolution of $M$ is supported in squarefree
degrees, see \cite{BrHe95}, \cite{Ya00}.

\begin{example} Let $A$ be the matrix of Example
\ref{basic_example} and let $\a=(1,0,1,1)$. Then
 $l_1=l_2=0$,
 \[C^\bullet (1,\a):\quad 0\mapright{} \Bbbk\mapright{}
 \Bbbk^2\mapright{}\Bbbk \mapright{}0\]
 \[ C^\bullet (2,\a):\quad 0\mapright{} \Bbbk\mapright{}
 \Bbbk^2\mapright{} 0\]
  and
  \[ C^\bullet (A,\a):\quad 0\mapright{} \Bbbk^2\mapright{}
 \Bbbk^4\mapright{}\Bbbk \mapright{}0\ . \]

It follows that  $\dim_{\Bbbk}H^0 (C^\bullet(A,\a))=1$ and
$b_{1,\a}(M_A)=1$ as expected.
\end{example}

\bigskip

\section{Local cohomology of $M$}

Let $A=(c_{ij}x^{\a_{ij}})$ be an $s\times l$
 multigraded squarefree matrix, $T$ a squarefree solution of
 $E_A$, $\phi=\phi_T$,
 $M_A=\coker  {\phi}$ and we let
$I_1,\ldots, I_s$ be the squarefree monomial ideals
 as in the previous section. We  proceed with the notation and related remarks.

\begin{notation}
\end{notation}
\begin{itemize}
\item{} Let $F\subset [n]$. Let $H$ be any $R$-module.
We write $H_F$ for the localization of $H$ at the powers of
$x^{\underline{F}}$. In particular for $F=\{ i_1,\ldots, i_t \}$,
$R_F=\Bbbk[x_1,\ldots, x_n, x_{i_1}^{-1},\ldots, x_{i_t}^{-1}]$.
Let $u\in H$. We write $u_F$ to denote the image of $u$ in $H_F$
under the natural homomorphism $H\mapright{} H_F$. If $\phi:
H_1\mapright{} H_2$ is an $R$-homomorphism we write $\phi_F:
(H_1)_F\mapright{} (H_2)_F$ for the induced homomorphism.
\item{} We let
$A_F= (c_{ij}{x^{\a_{ij}}_F})$. We recall that $A$ is the matrix
of $\phi: \ F_1\mapright{} F_0 $ with respect to  bases $\{ w_i:\
i=1,\ldots, l\}$  of $F_1$ and $\{v_j: \ j=1,\ldots, s\}$ of
$F_0$. Thus  $A_F$ is the matrix of $\phi_F: \ (F_1)_F\mapright{}
(F_0)_F $ with respect to the bases $\{ (w_i)_F:\ i=1,\ldots, l\}$
of $(F_1)_F$  and $\{(v_j)_F: \ j=1,\ldots, s\}$ of $(F_0)_F$ and
the multidegrees of $(w_i)_F$, $(v_j)_F$, $i\in [l]$, $j\in [s]$
are squarefree.  For each $j$ we let $I_{j,F}=(I_j)_F$.
\item{} Let $\D$ be a simplicial complex on $[n]$ and let $\a\in
\mathbb{Z}^n$. We let
\[\D^\a=\{ \t: \ \t\cap\s_{a^-}=\emptyset, \
\t\cup \s_\a\in \D\}\ .\]  We note that if $\a=-\a^-$ and
$\s_{\a^+}=\emptyset$ then $\D^\a$ is by definition the link of
$\s_{\a}$ in $\D$.
\item{} Let $F, G\subset [n]$. Let $N$ be an $R$-module. We let $\th_{F,G}: \
M_F\mapright{} M_G$, $\th_{F,G}(u_F)=u_G$  if $G=Fh$ and zero
otherwise. We let $K(x^\infty, N)$ to be the complex $$
K(x^\infty, N):\ 0\rightarrow N\mapright{\th^0}
\bigoplus_{\begin{array}{c} |F|=1\cr F\subset[n]\end{array}}
N_F\mapright{\th^1} \cdots\rightarrow N_{[n]} \rightarrow 0 $$
where $\th^r|_{N_F}=(\th_{F,G}) $. It is well known that for any
multigraded module $N$ and $\a\in \mathbb{Z}^n$,
\[ H^i_m(N)_\a=H^i( K(x^\infty, N)_\a)\ , \] see \cite{BrHe98}.
Moreover when $I$ is a squarefree monomial ideal, then by
reordering the variables of $R$ so that the indices of $\s_{\a^-}$
are at the end of $[n]$ one gets:
\begin{equation}\label{graded_local_cohomology}
K(x^\infty, R/I)_\a\cong
C^\bullet((\D_I)^\a)[-|\s_{\a^-}|-1]\end{equation}  and \[
\dim_{\Bbbk} H^i_m (R/I)_\a=\dim_{\Bbbk}
H^{i-|\s_{\a^-}|-1}(C^\bullet((\D_I)^\a))\ ,\]  see \cite{St83} or
\cite{BrHe98}.
\item{}  We recall from the
previous section that if $\a\in \mathbf{Z}^n$ then $\a_i=\a-\b_i$.
We denote by  $ \D^\a_j$ the complex $ (\D_{I_j})^{\a_j}$. Thus
\[\D_j^\a=\{ \t: \ \t\cap\s_{a_{j^-}}=\emptyset, \
\t\cup \s_{\a_j}\in \D_{I_j}\}\ .\] If $F\subset[n]$ we let
$B_{j,F}:=\{\overline{x^{\g} v_{j,F}}:\ x^{\g}\in R_F,\
x^{\g}\notin I_{j,F}\}$. We let
\[B_F(A)=\bigcup_j B_{j,F}\ .\]

\end{itemize}

\noindent We note that $B_{j,F}= \{\overline{x^\g v_{j,F}}:\
\s_{\g^-}\subset F,\ F\cup \s_{\g^+}  \in \D_{I_j}\}$. Moreover
$\deg(\overline{x^{\g} v_{j,F}})=\g+\b_j$. Thus the elements of
$B_F(A)$ of degree $\a$ form the set $B_F(A)_\a=B_F(A)\cap
\{\overline{x^{\g} v_{i,F}}:\ \g+\b_i=\a,\ i=1,\ldots, s\}= \{
\overline{x^{\a_i} v_{i}}_F:\ \s_{\a_i^-}\subset F,\
F\widehat{\s_{\a_i^-}}\in \D^\a_i,\ i=1,\ldots,s\} $. In the next
Theorem we determine a $\Bbbk$-basis for $K(x^\infty, M_A)^r_\a$
and its various homological components.

\begin{theorem}\label{k_basis_loc} Let  $A$ be as above. Then
\begin{enumerate}
\item{} $ B_F(A)$ is a multigraded $\Bbbk$-basis for $(M_A)_F$.
Moreover if $x^\g\in I_{i,F}$ then
\[\overline{x^{\g} v_{i,F}}=\sum_{j<i} r_{i,j,{\g^+} +
{\underline{F}}}\ \overline{x^{\g + \b_i - \b_j} v_{j,F}}\ .\]
\item{} Let $\a\in \mathbb{Z}^n$. The set \[B_{\a}(A)={\ds\bigcup_{F\subset [n]}} B_F(A)_\a\] is a
 $\Bbbk$-basis for $K(x^\infty, M_A)_\a$.
\item{} The set \[B_{\a,r}(A)=\{ \overline{x^{\a_i} v_{i}}_F: F\subset[n],
|F|=r,\ \s_{\a_i^-}\subset F,\ F\widehat{\s_{\a_i^-}}\in
\D^\a_i\}\] is a $\Bbbk$-basis for $K(x^\infty, M_A)^r_\a$.
\end{enumerate}
\end{theorem}
\begin{proof}
We prove the first claim. The rest follows by degree
consideration.
 First we prove
linear independence. Suppose that $x^\g \notin I_{j, F}$. Let
$\b\in \mathbb{N}^n$ such that $\s_{\g^-}\subset \s_{\b}\subset
F$. Then $x^\b x^\g\notin I_j$. Thus by clearing denominators, any
possible linear dependence relation on the elements of $B_F(A)$
corresponds to a linear dependence relation on the elements of
$B(M_A)$.

Next we show that $B_F(A)$ spans $(M_A)_F$. Let $\g\in
\mathbb{N}^n$, such that $x^\g$ is a generator of $I_{i,F}$. Then
$x^{\g^+}x^{\underline{F}}$ is a generator of $I_i$.  By Corollary
\ref{gen-coeff_basis} it follows that
\[\overline{x^{{\g^+}+{\underline{F}}}\ v_{i}}=\sum_{j<i} r_{i,j,{\g^+} +
{\underline{F}}}\ \overline{x^{\g^+ +\b_i-\b_j}\ v_{j}}\ .\]
Localizing at the powers of $x^{\underline{F}}$ and dividing by
$x^{\g^-}x^{\underline{F}}$ we get the desired claim.
\end{proof}

Next we describe the complex that will be used to compute
$H^i_m(M_A)_\a$. First we need to define one more sign: let
$\s\subset F$. We reorder $F$ so that the elements of $\s$ are at
the end of $F$. If the number of transpositions needed to do this
is even we let $t(\s,F):=1$, otherwise we let $t(\s,F):=-1$. If
$h\notin F$, it is direct to verify that
\[t(\s,F) \sgn[h, F h\hat{ \s}]=t(\s, Fh) \sgn[h, Fh]\ .\]

\begin{theorem}\label{local-cohom-com} Let $A$ be as above and   $\a\in \mathbb{Z}^n$. Let
$(L_j^\a)^\bullet=C^\bullet (\D^a_j)$ and
$l_j^-=|\s_{a_i^-}|-|\s_{a_1^-}|$. For each $r\in \mathbb{Z}$ we
let
\[(L^\a)^r:=\sum_{i=1}^r L_i^{r-{l_i^-}}\]
and
 $d^r: (L^\a)^r\mapright{} (L^\a)^{r+1}$  be  such that when
  $\t\in \D^\a_i$, $|\t|=r+1-{l_i^-}$
then
\[ \t^*\mapsto \sum_{h\in V_{\t, \D_{i}^\a, 1}  }\sgn[h, \t h]\ (\t h)^*
\]
\[+\ \sum_{{h\in V_{\t, \D_{i}^\a, 2}  }}\sgn[h, \t h]\ \sum_{ j<i } r_{i,j,
{\a_i^-+\underline{\t h}}}\ \frac{t(\s_{\a_i^-},\t h\
\s_{\a_i^-})}{t(\s_{\a_j^-},\t h \s_{\a_i^-})}, (\t h\
\s_{\a_i^-}\ \widehat{\s_{\a_{j}^-}})^*\  .\]
$ (L^\a(A))^\bullet,
d^\bullet)$ is a cochain complex and
\[\dim_{\Bbbk} H^i_m (M_A)_\a=\dim_{\Bbbk} H^{i-|\s_{\a_1^-}|-1}(L^\a(A))^\bullet).\]
\end{theorem}

\begin{proof}
There is an isomorphism of vector spaces
\[ K(x^\infty, M_A)^r_\a\cong \sum_i (L_i^\a)^{r-|\s_{a_i^-}|-1}\]
where
\[\overline{x^{\a_i} v_{i,F}}\mapsto
(F\widehat{ \s_{a_i^-}})^*\ .\] It is routine to show  that the
following diagram commutes:
\[ \begin{array}{ccc} \cr (L^{a})^t  &\mapright{d^t} &(L^{a})^{t+1}  \cr
\mapdown{} &
 &\mapdown{} \cr K(x^\infty, M_A)^{t+|\s_{a_1^-}|+1}_\a
 &\mapright{} &K(x^\infty, M_A)^{t+|\s_{a_1^-}|+2}_\a\ .\end{array}\]
\end{proof}

\begin{example} Let $A$ be the matrix of Example
\ref{basic_example} and let $\a=(0,-1,-1,0)$. Then
$\a_1=(0,-1,-1,-1)$, $\a_2=(-1,-1,-1,0)$, $l_1^-=l_2^-=0$.
Moreover $\D_1^\a=\D_2^\a=\{ \emptyset\}$,
\[ (L^\a)^\bullet:\quad 0\mapright{} {\Bbbk}^2
\mapright{}0\ ,\] and $ \dim_{\Bbbk} H^{-1}(L^\a)=2$. It follows
that $ \dim_{\Bbbk} H^{3}_m(M_A)_\a=2$. We do in more detail the
case for $\a=(0,0,0,0)$.  Here $\a_1=(0,0,0,-1)$,
$\a_2=(-1,0,0,0)$, $l_1^-=l_2^-=0$, $\s_{\a_1^-}=\{4\}$ and
$\s_{\a_2^-}=\{1'\}$. $\D_1^\a$ has facets the boundary of the
triangle $\{1,2,3\}$ while the facets of $\D_2^\a$ are $\{ 4\}$
and $\{2,3\}$. We have
\[
L^\a: \ 0\mapright{} \Bbbk^2 \longrightarrow \Bbbk^6
\longrightarrow \Bbbk^4 \mapright{} 0 \ \] with zero cohomology at
all homological degrees. For $\t=\{2\}\in \D_2^\a$, $V_{\t,
\D_2^\a,1}=\{3\}$, $V_{\t, \D_2^\a,2}=\{4\}$  and $
d^0(\t^*)=\t_1-\t_2$ where $\t_2=\{ 2, 3\}$ (in $\D_2^\a$) and
$\t_1=\{ 1,2\}$ (in $\D_1^\a$).
\end{example}

We finish this section with a corollary whose proof is immediate.
\begin{corollary} Let $A$  be as above. If for some $\a\in
\mathbb{Z}^n$, $H^i_m (M)_\a\neq 0$, then $\dim_{\Bbbk} H^i_m
(M_A)_\b=\dim_{\Bbbk} H^i_m (M_A)_\a$ for all $\b\in \mathbb{Z}^n$
such that $\s_{\b^+}=\s_{\a^+}$ and $\s_{\b^-}=\s_{\a^-}$.
\end{corollary}


\end{document}